\documentclass[a4paper,12pt]{article}
\usepackage[USenglish]{babel}
\usepackage[utf8]{inputenc}
\usepackage{amsthm,amsmath,amssymb,float,enumitem,caption,subcaption,mdframed,mathrsfs,soul}
\usepackage[a]{esvect}
\usepackage[dvipsnames]{xcolor}
\usepackage{tikz,pgf,pgfplots}
\pgfplotsset{compat=1.16}
\usetikzlibrary{decorations.pathreplacing}
\usetikzlibrary{patterns,patterns.meta,cd,angles,calc,quotes}
\usepackage[outline]{contour}
\contourlength{1.2pt}
\usepackage{geometry}
\usepackage[hidelinks]{hyperref}

\title{Path-Connectedness of the Hyperspace of Compact Subsets of $\mathbb{R}^n$}
\author{Bryant Rosado Silva$^\ast$ and Rodney Josué Biezuner$^\dagger$}
\date{}

\newcommand{\HE}[1]{\mathcal{H}\left(#1\right)}
\newcommand{\ol}[1]{\overline{#1}}
\newcommand{\Rn}{\mathbb{R}^n}

\theoremstyle{definition}
\newtheorem{definicao}{Definition}[subsection]
\newtheorem{exemplo}[definicao]{Example}
\newtheorem{observacao}[definicao]{Remark}

\theoremstyle{theorem}
\newtheorem{proposicao}[definicao]{Proposition}
\newtheorem{teorema}[definicao]{Theorem}
\newtheorem{lema}[definicao]{Lemma}

\renewenvironment{proof}{\noindent\textbf{Proof:}}{\hfill $\blacksquare$}

\newcommand\blfootnote[1]{%
  \begingroup
  \renewcommand\thefootnote{}\footnote{#1}%
  \addtocounter{footnote}{-1}%
  \endgroup
}

\definecolor{fifthblue}{HTML}{CF132C}
\definecolor{fourthblue}{HTML}{DB7F23}
\definecolor{thirdblue}{HTML}{C2C22D}
\definecolor{secondblue}{HTML}{52CC64}
\definecolor{firstblue}{HTML}{33A5C4}

\begin{document}
\maketitle
\blfootnote{$\ast$bryantrs99@ufmg.br (corresponding author)}
\blfootnote{$\dagger$rodneyjb@ufmg.br}

\vspace{-12pt}
\begin{center}
 \it Departamento de Matemática ICEx, Universidade Federal de Minas Gerais,\\
Av. Antônio Carlos 6627, Pampulha, CP 702, CEP 31270-901, Belo Horizonte, MG, Brazil
\end{center}

\begin{abstract}
 When one considers the collection $\HE{\Rn}$ of all compact subsets of $\Rn$ and equip it with a topology, many questions can be asked about the topological space one ends up with. This is an example of a hyperspace, a mathematical object which has been studied in a more abstract setting since the beginning of the 20th century. Here we give an elementary proof of the path-connectedness of $\HE{\Rn}$, with the topology induced by the Hausdoff metric, by exploring the vector structure of $\Rn$ and using only basic ideas of topology of metric spaces that undergraduate students with just a basic knowledge of these concepts will be able to understand.
\end{abstract}

\section{$\HE{X}$ and Hausdorff Metric}

\hspace{\parindent}Given a metric space $(X,d)$, take the collection of all non-empty compact subsets of $X$ and denote it by $\HE{X}$. When one provides a topology to this collection, it is called a \textbf{hyperspace}. In this text we will provide this collection with a metric, the Hausdorff metric, and use the topology induced by it. First we start with the usual idea of distance between a point and a set:

\begin{definicao}
Given a metric space $(X,d)$, let $B$ be a non-empty subset of $X$. If $x \in X$, the distance between $x$ and $B$ is defined as:
$$d(x,B)=\inf_{b \in B} d(x,b).$$
\hfill $\square$
\end{definicao}
Recall that if $B$ is compact, then the infimum is attained so one can just use the minimum.

This notion could be extended as usual to define the distance between subsets, taking the infimum over the first entry, but then the distance between every pair of distinct subsets that intersect each other would be zero which would do nothing to help us distinguish them. Worse, this would not give us a metric, so we will use the supremum instead.

\begin{definicao}
Given a metric space $(X,d)$, let $A$ and $B$ be points of $\HE{X}$, that is, compact and non-empty subsets of $X$. The distance from $A$ to $B$ is
$$\overline{d}(A,B)=\sup_{a \in A} d(a,B)=\max_{a \in A} d(a,B).$$
\hfill $\square$
\end{definicao}
Again, because of the compactness of $A$ and $B$, we have that there are points $\hat{x} \in A$ and $\hat{y} \in B$ such that $\overline{d}(A,B)=d(\hat{x},\hat{y})$.

It is important that this definition and what it implies is well understood by the reader, so we illustrate it with an example.

\begin{observacao}
 In every example we use the Euclidean metric and denote it by $d$.
\end{observacao}

\begin{exemplo}
Let $A$ be a line segment over the $x$-axis of $\mathbb{R}^2$ with length 1. Extend this line to the right, doubling the length, and shift it in the positive direction of the $y$-axis by one unit. Naming this new set as $B$, we have the following situation:

 \begin{figure}[H]
 \centering
  \begin{tikzpicture}[line width=1pt]
   \draw (0,0)--(1,0) node[below right] {$A$};
   \draw (0,1)--(2,1) node[above right] {$B$};
   \draw[|-|,xshift=-6pt,line width = 0.5pt] (0,0)--(0,1);
   \node[left,xshift=-6pt] at (0,0.5) {$1$};
   \draw[|-|,yshift=6pt,line width = 0.5pt] (2,1)--(0,1);
   \node[above,yshift=6pt] at (1,1) {$2$};
   \draw[|-|,yshift=-6pt,line width = 0.5pt] (1,0)--(0,0);
   \node[below,yshift=-6pt] at (0.5,0) {$1$};
  \end{tikzpicture}
  \caption{Line segments $A$ and $B$.}
 \end{figure}
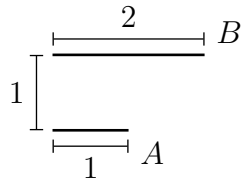
 Let us calculate $\overline{d}(A,B)$ and $\overline{d}(B,A)$:
 \begin{itemize}[label={$\bullet$}]
  \item $\overline{d}(A,B)$:

  By definition,
  $$\overline{d}(A,B)= \max_{a \in A} \{d(a,B)\}= \max_{a \in A} \left\{\min_{b \in B} d(a,b)\right\}.$$
  Fixing a point $a$ of $A$, the point $b$ of $B$ that is nearer to $a$, hence the one that realizes $d(a,B)$, is the point in the intersection of $B$ and the orthogonal line to $B$ that passes through $a$, whose length is 1.

  \begin{figure}[H]
 \centering
  \begin{tikzpicture}[line width=1pt]
   \draw (0,0)--(1,0) node[below right] {$A$};
   \draw (0,1)--(2,1) node[above right] {$B$};
   \draw[dash pattern = on 1.5pt off 1.5pt,] (0.5,0)--(0.5,1);
   \draw[fill=black] (0.5,0) circle (1.5pt) node[below] {$a$};
   \draw[fill=black] (0.5,1) circle (1.5pt) node[above] {$b$};
   \draw[line width=0.5pt] (0.5,0.8)--(0.7,0.8)--(0.7,1);
   \draw[fill=black] (0.6,0.9) circle (0.5pt);
  \end{tikzpicture}
  \caption{Calculating $d(a,B)$.}
 \end{figure}
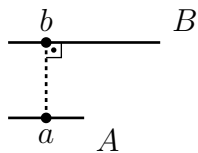
 Therefore, for all $a \in A$ we have that
 $$\min_{b \in B} d(a,b) = 1,$$
 and this means that $$\ol{d}(A,B)=1.$$

 \item $\overline{d}(B,A)$:

 In the same fashion,
 $$\overline{d}(B,A)= \max_{b \in B} \{d(b,A)\}= \max_{b \in B} \left\{\min_{a \in A} d(b,a)\right\}.$$

 This is the case that draws our attention to the fact that we still do not have a metric since $\overline{d}$ is not symmetric. Indeed, fixed $b$, there are two possibilities: either there exists a line passing through $b$ and orthogonal to $A$ or there is not such a line. In the first case, as we did when calculating $\overline{d}(A,B)$, we have $d(b,A)=1$, while in the other case we have $d(b,A)>1$ by the triangle inequality. In fact, this happens to every point of $b$ with $x$ coordinate greater than $1$ and in this situation the point $a \in A$ that minimizes the distance is the right end of $A$.

 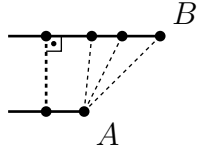
\begin{figure}[H]
 \centering
  \begin{tikzpicture}[line width=1pt]
   \draw (0,0)--(1,0) node[below right] {$A$};
   \draw (0,1)--(2,1) node[above right] {$B$};
   \draw[dash pattern = on 1.5pt off 1.5pt,] (0.5,0)--(0.5,1);
   \draw[fill=black] (0.5,0) circle (1.5pt);
   \draw[fill=black] (0.5,1) circle (1.5pt);
   \draw[line width=0.5pt] (0.5,0.8)--(0.7,0.8)--(0.7,1);
   \draw[fill=black] (0.6,0.9) circle (0.5pt);
   \draw[fill=black] (1.1,1) circle (1.5pt);
   \draw[fill=black] (1.5,1) circle (1.5pt);
   \draw[fill=black] (2,1) circle (1.5pt);
   \draw[fill=black] (1,0) circle (1.5pt);
   \draw[dash pattern= on 1.5pt off 1.5 pt, line width = 0.5pt] (1.1,1)--(1,0);
   \draw[dash pattern= on 1.5pt off 1.5 pt, line width = 0.5pt] (1.5,1)--(1,0);
   \draw[dash pattern= on 1.5pt off 1.5 pt, line width = 0.5pt] (2,1)--(1,0);
  \end{tikzpicture}
  \caption{Illustration of the two cases when calculating $\overline{d}(B,A)$.}
 \end{figure}
 Since we want the maximum of the distances over $B$, if we take the right end of $B$, we end up with $$\overline{d}(B,A)=\sqrt{2}$$
 \end{itemize}
and thus $\overline{d}(A,B)= 1 \ne \sqrt{2} = \overline{d}(B,A)$.
 \hfill $\square$
\end{exemplo}

There is a simple way to solve this problem which will lead us to the already mentioned (but not yet defined) Hausdorff metric:

\begin{proposicao}[Hausdorff metric]
Given a metric space $(X,d)$, let $A$ and $B$ be elements of $\HE{X}$. The function $h: X \times X \to \mathbb{R}$ given by
$$h(A,B)=\max\{\overline{d}(A,B),\overline{d}(B,A)\}$$
is a metric on $\HE{X}$.
\end{proposicao}

\begin{proof}
To show that $h$ is a metric, we need to prove that given $A,B$ and $C$ in $\HE{X}$, we have
\begin{enumerate}[label={(\roman*)}]
 \item $0\leq h(A,B) < \infty$;
 \item $h(A,B)=0$ if and only if $A=B$;
 \item $h(A,B)=h(B,A)$;
 \item $h(A,B) \leq h(A,C)+h(C,B)$.
\end{enumerate}

From the definition of $h$ it can be immediately seen that (i) and (iii) are satisfied. In order to prove (ii), notice that if $A=B$, for every $a \in A$ we have $\overline{d}(a,A)=0$, thus $h(A,A)=0$. If $h(A,B)=0$, assume by contradiction that $A \ne B$. With this assumption,
$$d(a,B)=0 \qquad \text{and} \qquad d(b,A)=0$$
for all $a \in A$ and $b \in B$. This implies that, $A \subset B$ and $B \subset A$ since $d(x,y)=0$ if and only if $x=y$. Therefore, $A=B$, a contradiction.

For the triangle inequality, take $c \in C$ such that $d(a,C)=d(a,c)$ and note that
\begin{align*}
 d(a,B)&= \min_{b \in B} d(a,b) \\
 &\leq \min_{b \in B} \{d(a,c)+d(c,b)\} \\
 &= d(a,c) + \min_{b \in B} d(c,b) \\
 &\leq d(a,C) + \ol{d}(C,B).
\end{align*}
So, choosing the point $a \in A$ satisfying $\overline{d}(A,B)=d(a,B)$, we end up with
\begin{align*}
 \ol{d}(A,B)=d(a,B)&\leq d(a,C) + \ol{d}(C,B)\\
 &\leq \ol{d}(A,C) + \ol{d}(C,B).
\end{align*}
Thus, doing the same thing for $\overline{d}(B,A)$, we conclude that
\begin{align*}
 h(A,B)&=\max \{\ol{d}(A,B),\ol{d}(B,A)\}\\
 &\leq \max \{\ol{d}(A,C)+\ol{d}(C,B),\ol{d}(B,C)+\ol{d}(C,A)\} \\
 &\leq \max \{\ol{d}(A,C),\ol{d}(C,A)\}+ \max\{\ol{d}(B,C),\ol{d}(C,B)\}\\
 &=h(A,C)+h(C,B).
\end{align*}

\end{proof}

This metric space $(\HE{X},h)$, called by Barnsley ``the space where fractals live''\cite[pp.~27]{Barn}, is an example of a hyperspace and this metric is the so-called Hausdorff metric or distance, and appeared first in Hausdorff's book \textit{Grundzüge der Mengenlehre} in 1914.

Before moving on to the next section, let us consider an important example that will help a lot in the understanding of our main result.

\begin{exemplo}
Consider the rectangular curves below, denoted by $A$ and $B$ respectively. We will compute $h(A,B)$ keeping always in mind that the sets $A$ and $B$ are just the boundaries of the rectangles, excluding their interiors.

 
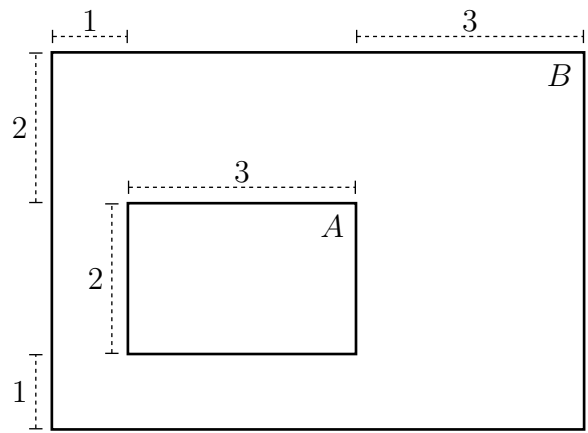
\begin{figure}[H]
 \centering
  \begin{tikzpicture}
   \draw[line width=1pt] (0,0) rectangle (7,5) node[below left] {$B$};
   \draw[line width=1pt] (1,1) rectangle (4,3) node[below left] {$A$};
   \draw[|-|,line width=0.5pt,dash pattern=on 1.5pt off 1.5pt,xshift=-6pt]
   (0,0)--(0,1);
   \draw[|-|,line width=0.5pt,dash pattern=on 1.5pt off 1.5pt,xshift=-6pt] (1,1)--(1,3);
   \draw[|-|,line width=0.5pt,dash pattern=on 1.5pt off 1.5pt,xshift=-6pt] (0,3)--(0,5);
   \node[xshift=-12pt] at (0,0.5) {$1$};
   \node[xshift=-12pt] at (0,4) {$2$};
   \node[xshift=-12pt] at (1,2) {$2$};
   \draw[|-|,line width=0.5pt,dash pattern=on 1.5pt off 1.5pt,yshift=6pt] (0,5)--(1,5);
   \draw[|-|,line width=0.5pt,dash pattern=on 1.5pt off 1.5pt,yshift=6pt] (1,3)--(4,3);
   \draw[|-|,line width=0.5pt,dash pattern=on 1.5pt off 1.5pt,yshift=6pt] (4,5)--(7,5);
   \node[yshift=12pt] at (0.5,5) {$1$};
   \node[yshift=12pt] at (2.5,3) {$3$};
   \node[yshift=12pt] at (5.5,5) {$3$};
  \end{tikzpicture}
 \end{figure}

 Let us start computing $\ol{d}(A,B)$. Since we are dealing with rectangular curves, it is possible to split $A$ in sets according to the distance of the point from $B$ by inspection without much trouble, as shown in the figure below:
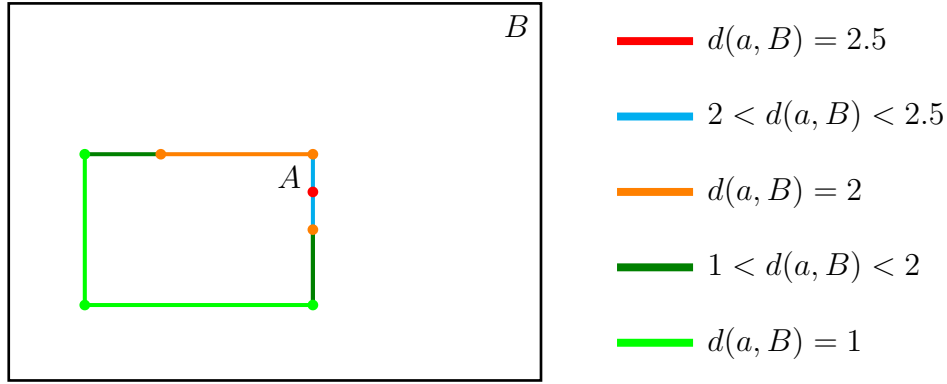
\begin{figure}[H]
 \centering
  \begin{tikzpicture}
   \draw[line width=1pt] (0,0) rectangle (7,5) node[below left] {$B$};
   \draw[line width=1pt] (1,1) rectangle (4,3) node[below left] {$A$};
   \draw[color=green,line width=1.5pt] (4,1)--(1,1)--(1,3);
   \draw[color=green!50!black,line width=1.5pt] (4,1)--(4,2);
   \draw[color=green!50!black,line width=1.5pt] (4,1)--(4,2);
   \draw[color=cyan,line width=1.5pt] (4,2)--(4,3);
   \draw[color=green!50!black,line width=1.5pt] (1,3)--(2,3);
   \draw[color=orange,line width=1.5pt] (2,3)--(4,3);
   \draw[color=orange,fill=orange,line width=1pt] (4,2) circle (1.5pt);
   \draw[color=red,fill=red,line width=1pt] (4,2.5) circle (1.5pt);
   \draw[color=orange,fill=orange,line width=1pt] (4,3) circle (1.5pt);
   \draw[color=orange,fill=orange,line width=1pt] (2,3) circle (1.5pt);
   \draw[color=green,fill=green,line width=1pt] (1,3) circle (1.5pt);
   \draw[color=green,fill=green,line width=1pt] (4,1) circle (1.5pt);
   \draw[color=green,fill=green,line width=1pt] (1,1) circle (1.5pt);
   \draw[color=green,line width=3pt] (8,0.5)--(9,0.5) node[right,color=black] {$d(a,B)=1$};
   \draw[color=green!50!black,line width=3pt] (8,1.5)--(9,1.5) node[right,color=black] {$1<d(a,B)<2$};
   \draw[color=orange,line width=3pt] (8,2.5)--(9,2.5) node[right,color=black] {$d(a,B)=2$};
   \draw[color=cyan,line width=3pt] (8,3.5)--(9,3.5) node[right,color=black] {$2<d(a,B)<2.5$};
   \draw[color=red,line width=3pt] (8,4.5)--(9,4.5) node[right,color=black] {$d(a,B)=2.5$};
  \end{tikzpicture}
 \end{figure}
 Therefore, $\ol{d}(A,B)=2.5$. We do the same to compute $\ol{d}(B,A)$:
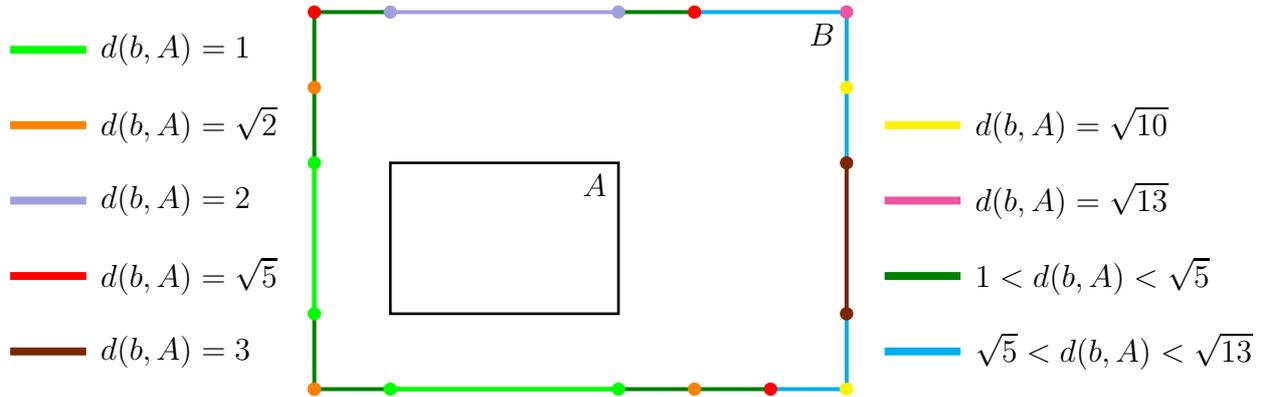
 \begin{figure}[H]
 \centering
  \begin{tikzpicture}
   \draw[line width=1pt] (0,0) rectangle (7,5) node[below left] {$B$};
   \draw[line width=1pt] (1,1) rectangle (4,3) node[below left] {$A$};
   \draw[color=green,line width=1.5pt] (0,1)--(0,3) (1,0)--(4,0);
   \draw[color=green!50!black,line width=1.5pt] (1,0)--(0,0)--(0,1) (0,3)--(0,5)--(1,5) (4,0)--(6,0) (4,5)--(5,5);
   \draw[color=green,line width=1.5pt] (0,1)--(0,3) (1,0)--(4,0);
   \draw[color=cyan,line width=1.5pt] (6,0)--(7,0)--(7,1) (7,3)--(7,5)--(5,5);
   \draw[color=Brown,line width=1.5pt] (7,1)--(7,3);
   \draw[color=blue!50!gray!50!white,line width=1.5pt] (1,5)--(4,5);
   \draw[color=orange,fill=orange,line width=1pt] (0,0) circle (2pt);
   \draw[color=orange,fill=orange,line width=1pt] (5,0) circle (2pt);
   \draw[color=orange,fill=orange,line width=1pt] (0,4) circle (2pt);
   \draw[color=orange,fill=orange,line width=1pt] (0,0) circle (2pt);
   \draw[color=green,fill=green,line width=1pt] (0,1) circle (2pt);
   \draw[color=green,fill=green,line width=1pt] (1,0) circle (2pt);
   \draw[color=green,fill=green,line width=1pt] (0,3) circle (2pt);
   \draw[color=green,fill=green,line width=1pt] (4,0) circle (2pt);
   \draw[color=red,fill=red,line width=1pt] (0,5) circle (2pt);
   \draw[color=red,fill=red,line width=1pt] (6,0) circle (2pt);
   \draw[color=yellow,fill=yellow,line width=1pt] (7,0) circle (2pt);
   \draw[color=yellow,fill=yellow,line width=1pt] (7,4) circle (2pt);
   \draw[color=red,fill=red,line width=1pt] (5,5) circle (2pt);
   \draw[color=blue!50!gray!50!white,fill=blue!50!gray!50!white,line width=1pt] (4,5) circle (2pt);
   \draw[color=blue!50!gray!50!white,fill=blue!50!gray!50!white,line width=1pt] (1,5) circle (2pt);
   \draw[color=Brown,fill=Brown,line width=1pt] (7,3) circle (2pt);
   \draw[color=Brown,fill=Brown,line width=1pt] (7,1) circle (2pt);
   \draw[color=Rhodamine,fill=Rhodamine,line width=1pt] (7,5) circle (2pt);
   \draw[color=green,line width=3pt] (-4,4.5)--(-3,4.5) node[right,color=black] {$d(b,A)=1$};
   \draw[color=green!50!black,line width=3pt] (7.5,1.5)--(8.5,1.5) node[right,color=black] {$1<d(b,A)<\sqrt{5}$};
   \draw[color=orange,line width=3pt] (-4,3.5)--(-3,3.5) node[right,color=black] {$d(b,A)=\sqrt{2}$};
   \draw[color=cyan,line width=3pt] (7.5,0.5)--(8.5,0.5) node[right,color=black] {$\sqrt{5}<d(b,A)<\sqrt{13}$};
   \draw[color=red,line width=3pt] (-4,1.5)--(-3,1.5) node[right,color=black] {$d(b,A)=\sqrt{5}$};
   \draw[color=Brown,line width=3pt] (-4,0.5)--(-3,0.5) node[right,color=black] {$d(b,A)=3$};
   \draw[color=yellow,line width=3pt] (7.5,3.5)--(8.5,3.5) node[right,color=black] {$d(b,A)=\sqrt{10}$};
   \draw[color=blue!50!gray!50!white,line width=3pt] (-4,2.5)--(-3,2.5)  node[right,color=black] {$d(b,A)=2$};
   \draw[color=Rhodamine,line width=3pt] (7.5,2.5)--(8.5,2.5) node[right,color=black] {$d(b,A)=\sqrt{13}$};
  \end{tikzpicture}
 \end{figure}
 So $\ol{d}(B,A)=\sqrt{13}$ and consequently $h(A,B)=\sqrt{13}$.

 \hfill $\square$
\end{exemplo}

\begin{observacao}
In other texts the hyperspace of compact subsets of $X$ is denoted as $\mathcal{K}(X)$ or $2^X$ as in \cite{Nad}.
\end{observacao}

\section{$\HE{\mathbb{R}^n}$ is path-connected}
\begin{definicao}
If $a$ and $b$ are two elements of a topological space $X$, a path from $a$ to $b$ is a continuous function
 \begin{align*}
 f:&[0,1] \to X, \text{ such that } f(0)=a \text{ and } f(1)=b.
 \end{align*}
 If there exist a path for every pair of elements of $X$, then $X$ is called \textbf{pathwise connected}.

 \hfill $\square$
\end{definicao}

It is interesting to note that the hyperspace inherits properties from its base space. For example, if $(X,d)$ is complete, then $(\HE{X},h)$ will be too \cite[pp.35]{Barn} and the same happens when the base space is compact. Now, what about connectedness? In \cite[pp.113]{Nad} it is proven that if $(X,d)$ is a \textit{continuum}, meaning a compact and connected metrizable space, then it will be more than pathwise connected, it will be arcwise connected which means that there is a path between any pair of elements and the inverse function of this path exists and is continuous (that is, there exists a path between the elements that contains no self-intersections). This is sufficient to establish that $\HE{\mathbb{R}^n}$ is pathwise connected, since every pair of compact sets of $\mathbb{R}^n$ is contained in a continuum. Here we are going to use a different and simpler approach (in the sense that nothing new is needed) to prove that $\HE{\mathbb{R}^n}$ is pathwise connected, inspired by the proof of the path connectedness of $\HE{\mathbb{R}}$ in \cite[pp.38-40]{Barn}.

\subsection{Paths in $\HE{\Rn}$}
From now on we will assume that $\mathbb{R}^n$ is equipped with the euclidean metric and by $\HE{\mathbb{R}^n}$ we mean $(\HE{\mathbb{R}^n},h)$. Our goal is to prove that given two elements $A$ and $B$ in $\HE{\Rn}$ there is a path from $A$ to $B$, that is, there exists $$f:[0,1] \to \HE{\Rn}$$ continuous such that $f(0)=A$ and $f(1)=B$.

A natural question is: can we interpret geometrically these paths in $\HE{\Rn}$? Since we want something continuous in the space of compact sets, does this mean that something happens to $A$ while the parameter changes? Well, yes, continuity has a fundamental role. If $f$ as above is continuous, given $t_0 \in [0,1]$ and $\varepsilon>0$, there is a $\delta>0$ such that if $t \in (t_0-\delta,t_0+\delta)\cap[0,1]$ then
$$h(f(t_0),f(t))< \varepsilon.$$
This means that any point of $f(t)$ is not further than $\varepsilon$ from some point of $f(t_0)$.

So, we can think of paths as deformations of $A$ into $B$ and fortunately, since we are working in $\Rn$, we can have a better idea because it is possible to visualize $\mathbb{R}^2$. Trying to make this as clear as possible, we are going to check some examples visualizing in $\mathbb{R}^2$ paths that are in $\HE{\mathbb{R}^2}$. The first one will make it easier to understand the proof of our main result.

\begin{exemplo}
Let $A$ be the set consisting of the single point $(a_x, a_y)$ and $B$ be the rectangle $[c_1,c_3] \times [c_2,c_4]$. We deform $A$ into $B$ through the following function $f:[0,1] \to \HE{\mathbb{R}^2}$ that will work as a path from $A$ to $B$ in $\HE{\mathbb{R}^2}$:
\begin{align*}
f(t) = [a_x+t(c_1-a_x),a_x+t(c_3-a_x)]\times[a_y+t(c_2-a_y),a_y+t(c_4-a_y)].
\end{align*}
Below, we illustrate the evolution of $f$ for some values of $t$.

 \begin{figure}[H]
 \centering
 \begin{tikzpicture}
  \foreach \t/\c in {1/green!50!black,0.75/brown,0.5/cyan,0.25/orange}{
  \draw[line width=0.5pt] (-5*\t,1-3*\t) rectangle (4*\t,1+2*\t);
  }
  \draw[fill=black] (0,1) circle (1.5pt) node[left] {$A$};
  \node[right] at (4,-1.9) {$B$};
  \node[above left] at (4,-2) {$t=1$};
  \node[above left] at (3,-1.25) {$t=0.75$};
  \node[above left] at (2,-0.5) {$t=0.5$};
  \node[below left] at (1,0.75) {$t=0.25$};
 \end{tikzpicture}
 \caption{Illustration of how we can interpret a path from $A$ to $B$.}
 \end{figure}
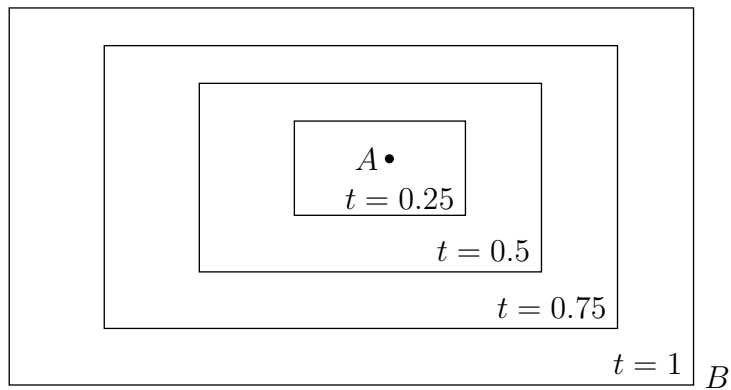

 \hfill $\square$
\end{exemplo}

This is a quite simple example (although not the simplest!) and we could do much more elaborated examples:

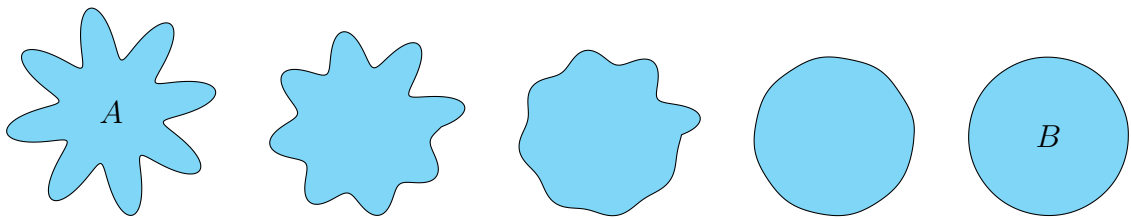
\begin{figure}[H]
 \centering
\begin{tikzpicture}
\draw[domain=0:360,samples=200,scale=0.35,fill=cyan,fill opacity=0.5] plot[smooth] (\x:{3+sin(8*\x)});
\node at (0,0) {$A$};
\end{tikzpicture}
\hspace{12pt}
\begin{tikzpicture}
\draw[domain=0:360,samples=200,scale=0.35,fill=cyan,fill opacity=0.5] plot[smooth] (\x:{3+sin(8*\x)/(1+\x*0.005)});
\end{tikzpicture}
\hspace{12pt}
\begin{tikzpicture}
\draw[domain=0:360,samples=200,scale=0.35,fill=cyan,fill opacity=0.5] plot[smooth] (\x:{3+sin(8*\x)/(1+\x*0.03)});
\end{tikzpicture}
\hspace{12pt}
\begin{tikzpicture}
\draw[domain=0:360,samples=200,scale=0.35,fill=cyan,fill opacity=0.5] plot[smooth] (\x:{3+sin(8*\x)/(25+\x*0.1)});
\end{tikzpicture}
\hspace{12pt}
\begin{tikzpicture}
\draw[domain=0:360,samples=200,scale=0.35,fill=cyan,fill opacity=0.5] plot[smooth] (\x:{3+sin(8*\x)/(100+\x)});
\node at (0,0) {$B$};
\end{tikzpicture}
\end{figure}

The most interesting thing here is that these paths can start with a pathwise connected set and end up with a discrete set or vice-versa, it can create holes or fill some, there are a lot of possibilities. Let's check more three examples, now with some description.

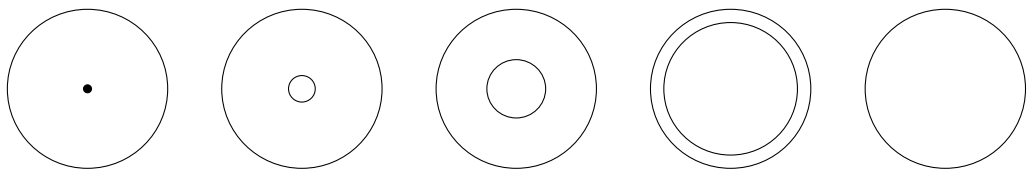
\begin{figure}[H]
\centering
\begin{subfigure}{\textwidth}
\centering
\begin{tikzpicture}
 \draw[fill=black] (0,0) circle (1.5pt);
 \draw (0,0) circle (30pt);
\end{tikzpicture}
\hspace{12pt}
\begin{tikzpicture}
 \draw (0,0) circle (5pt);
 \draw (0,0) circle (30pt);
\end{tikzpicture}
\hspace{12pt}
\begin{tikzpicture}
 \draw (0,0) circle (11pt);
 \draw (0,0) circle (30pt);
\end{tikzpicture}
\hspace{12pt}
\begin{tikzpicture}
 \draw (0,0) circle (25pt);
 \draw (0,0) circle (30pt);
\end{tikzpicture}
\hspace{12pt}
\begin{tikzpicture}
 \draw (0,0) circle (30pt);
\end{tikzpicture}
\subcaption{Set in $\mathbb{R}^2$ composed of a circumference and a dot that is taken to the same circumference by the path $f$ defined in $[0,1]$ as $f(t)=S(0;1)\cup S(0;t)$.}
\end{subfigure}
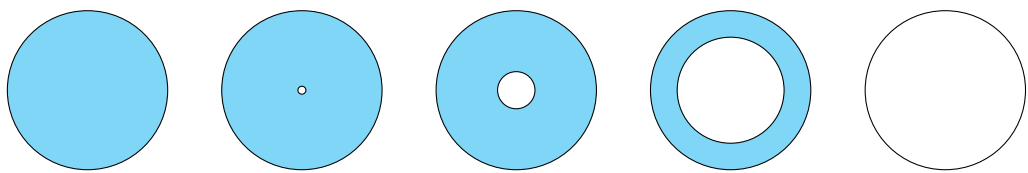
\vspace*{12pt}

\begin{subfigure}{\textwidth}
\centering
\begin{tikzpicture}
 \draw[fill=cyan,fill opacity=0.5] (0,0) circle (30pt);
\end{tikzpicture}
\hspace{12pt}
\begin{tikzpicture}
 \draw[fill=cyan,fill opacity=0.5,even odd rule]
 (0,0) circle (1.5pt)
 (0,0) circle (30pt);
\end{tikzpicture}
\hspace{12pt}
\begin{tikzpicture}
 \draw[fill=cyan,fill opacity=0.5,even odd rule]
 (0,0) circle (7pt)
 (0,0) circle (30pt);
\end{tikzpicture}
\hspace{12pt}
\begin{tikzpicture}
\draw[fill=cyan,fill opacity=0.5,even odd rule]
 (0,0) circle (20pt)
 (0,0) circle (30pt);
 \end{tikzpicture}
 \hspace{12pt}
\begin{tikzpicture}
 \draw (0,0) circle (30pt);
\end{tikzpicture}
\subcaption{Unitary disc deformed into a circumference by the path $f$ defined in $[0,1]$ as $f(t)=\overline{B(0;1)} \setminus B(0;t)$}
\end{subfigure}
\end{figure}
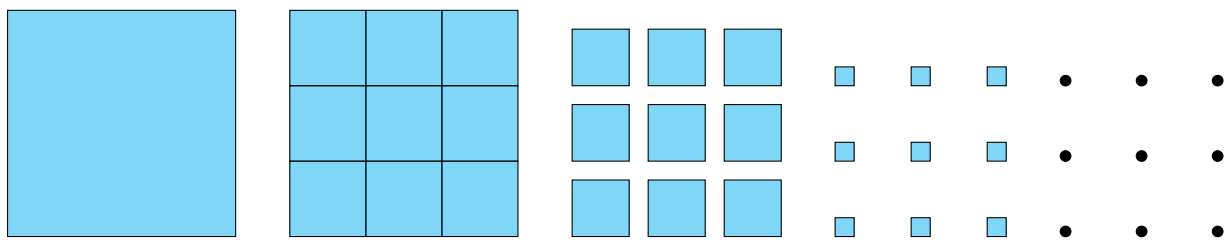
\vspace*{-12pt}
\begin{figure}[H]
 \continuedfloat
 \begin{subfigure}{\textwidth}
\centering
\begin{tikzpicture}
 \draw[fill=cyan,fill opacity=0.5] (0,0) rectangle (3,3);
\end{tikzpicture}
\hspace{12pt}
\begin{tikzpicture}
\foreach \i in {0,1,2}{
\foreach \j in {0,1,2}{
 \draw[fill=cyan,fill opacity=0.5] (\i,\j) rectangle (\i+1,\j+1);
 }
 }
\end{tikzpicture}
\hspace{12pt}
\begin{tikzpicture}
\foreach \i in {0,1,2}{
\foreach \j in {0,1,2}{
 \draw[fill=cyan,fill opacity=0.5] (\i,\j) rectangle (\i+0.75,\j+0.75);
 }
 }
\end{tikzpicture}
\hspace{12pt}
\begin{tikzpicture}
\foreach \i in {0,1,2}{
\foreach \j in {0,1,2}{
 \draw[fill=cyan,fill opacity=0.5] (\i,\j) rectangle (\i+0.25,\j+0.25);
 }
 }
\end{tikzpicture}
\hspace{12pt}
\begin{tikzpicture}
\foreach \i in {0,1,2}{
\foreach \j in {0,1,2}{
 \draw[fill=black] (\i,\j) circle (2pt);
 }
 }
\end{tikzpicture}
\subcaption{A square can be ``transformed'' into a discrete set composed by nine points. To have such a path, describe the square as a union of nine Cartesian products and use ideas explored until now, shrinking the intervals as $t$ increases.}
\end{subfigure}
\end{figure}

\subsection{Idea of the Proof}
We are going to start with two definitions that will simplify the process. We will denote the coordinates of a point $x \in \Rn$ by $x=(x^1,x^2,\ldots,x^n)$.

\begin{definicao}
Let $x,y \in \Rn$. The points $x$ and $y$ define the rectangle
$$\prod_{i=1}^n[\min\{x^i,y^i\},\max\{x^i,y^i\}].$$
Indeed, $x$ and $y$ are vertices of this rectangle.

\hfill $\square$
\end{definicao}

\begin{observacao}
Pay attention to the fact that here a rectangle need not to be a $n$-dimensional object. It can even be a point or a $3$-dimensional cube if $n \geq 3$. Everything depends on the coordinates of $x$ and $y$.
\end{observacao}

\begin{figure}[H]
\centering
 \begin{tikzpicture}
  \draw (0,0) rectangle (4,2);
  \draw[fill=black] (0,2) circle (2pt) node[above right] {$x_1=(0,2)$};
  \draw[fill=black] (4,0) circle (2pt) node[below left] {$y_1=(4,0)$};
 \end{tikzpicture}
 \hspace{36pt}
 \begin{tikzpicture}
  \draw (0,0) rectangle (4,2);
  \draw[fill=black] (0,0) circle (2pt) node[below right] {$x_2=(0,0)$};
  \draw[fill=black] (4,2) circle (2pt) node[above left] {$y_2=(4,2)$};
 \end{tikzpicture}
 \caption{Two ways to define the same rectangle, using the pairs of vertices $x_1,y_1$ or $x_2,y_2$.}
\end{figure}
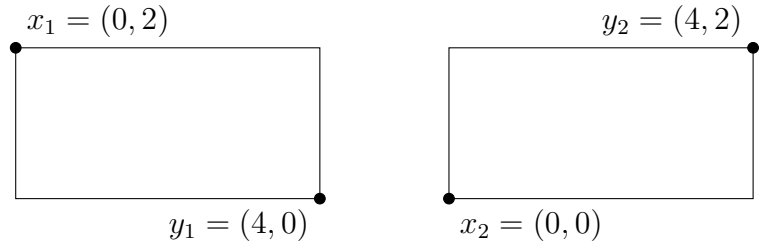

\begin{observacao}\label{retn}
 Note, as shown in the last figure, that there is more than one choice of pairs of vertices that define the same rectangle $R$. For any given rectangle $R$, we will choose as its vertices the ones that satisfy
 $$x^i = \min\{x^i,y^i\}, \quad y^i = \max \{x^i,y^i\}$$
 and
 $$R = \prod_{i=1}^n[x^i,y^i].$$
\end{observacao}

The next definition relies on the vector structure of $\Rn$ and will allow us to define a special class of paths.

\begin{definicao}
Let $A$ be a subset of $\Rn$ and $\vv{v} \in \Rn$. The translation of $A$ by $\vv{v}$ is denoted by
$$A+\vv{v} = \{a+\vv{v}: a \in A\}.$$

\hfill $\square$
\end{definicao}

Having understood what is a path in $\HE{\Rn}$ and the definitions above, we can start discussing the proof of path connectedness of $\HE{\Rn}$. It is simple but a bit long, so we split it in three parts:
\begin{enumerate}[label={\arabic*)}]
 \item Show that any translation of a set is pathwise connected to the original set. In particular, this means that any one point sets are pathwise connected.
 \item Show that any one point set is pathwise connected to any rectangle which contains the set.
 \item Show that any compact set is pathwise connected to any rectangle which contains it.
\end{enumerate}

Therefore, if we have two compact sets $A$ and $B$, there are rectangles $A_R$ and $B_R$ that contains $A$ and $B$ respectively and we can take $A_R$ and $B_R$ such that one is a translation of the other one. Doing this, we just need to juxtapose the paths from $A$ to $A_R$, $A_R$ to $B_R$ and $B_R$ to $B$, ending up with a path from $A$ to $B$ which concludes the proof of the path-connectedness of $\HE{\Rn}$.

\subsection{Proof}

\begin{lema} \label{translationpath}
Let $A \in \HE{\Rn}$ and $\vv{v} \in \Rn$. The map
 \begin{align*}
  f:& [0,1] \to \HE{\Rn} \\
  & t \mapsto A+ t\vv{v}
 \end{align*}
 is a path from $A$ to $A + \vv{v}$. This is called the \textbf{translation path}.
\end{lema}

 \begin{proof}
Clearly $f(0)=A$ and $f(1)=A+\vv{v}$. Moreover, $f$ is well defined because translations do not affect compactness. It remains to show that $f$ is continuous, but we will do more than this, we will prove that it is uniformly continuous. When $\vv{v}=\vv{0}$, $f$ is the constant path, so it is uniformly continuous.
 Assume for every $a \in A$ we have
 \begin{align*}
  d(a+t_1\vv{v},A+t_2\vv{v}) \leq d(a+t_1\vv{v},a+t_2\vv{v}) = |t_1\vv{v} - t_2\vv{v}|= |t_1 - t_2||\vv{v}| < \delta |\vv{v}|.
 \end{align*}
 Let $\varepsilon>0$ be given. Setting $\delta= \dfrac{\varepsilon}{|\vv{v}|}$, if $|t_1-t_2|<\delta$, we get $d(a+t_1\vv{v},A+t_2\vv{v}) < \varepsilon$. Thus,
 $$\ol{d}(A+t_1 \vv{v},A+t_2\vv{v}) = \max_{a \in A} d(a+t_1\vv{v},A+t_2\vv{v}) < \varepsilon.$$
 Similarly, if $|t_1 - t_2|<\delta$ then we also have $\ol{d}(A+t_2\vv{v},A+t_1\vv{v})< \varepsilon$. This proves uniform continuity because if $\delta< \dfrac{\varepsilon}{|\vv{v}|}$ then
\begin{align*}
h(f(t_1),f(t_2))&=\max \{\ol{d}(f(t_1),f(t_2)),\ol{d}(f(t_2),f(t_1))\} \\
&= \max\{\ol{d}(A+t_1\vv{v},A+t_2\vv{v}),\ol{d}(A+t_2\vv{v},A+t_1\vv{v})\}<\varepsilon.
\end{align*}

\end{proof}

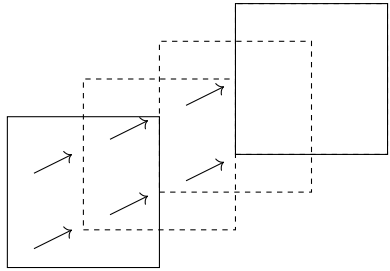
\begin{figure}[H]
\centering
\begin{tikzpicture}[scale=1]
\foreach \x in {0,1,2,3}{
\pgfmathtruncatemacro\myresult{0<\x?1:0}
\pgfmathtruncatemacro\myresultt{\x<3?1:0}
\ifnum\myresult>0\relax%
	\draw[dash pattern= on 2pt off 2pt] (\x,0+\x*0.5) rectangle (\x+2,2+\x*0.5);
\else%
	\draw (\x,0+\x*0.5) rectangle (\x+2,2+\x*0.5);
\fi
\ifnum\myresultt>0\relax%
\else%
	\draw (\x,0+\x*0.5) rectangle (\x+2,2+\x*0.5);
\fi
\pgfmathtruncatemacro\myresult{\x<3?1:0}
\ifnum\myresult>0\relax%
	\foreach \y in {0.25,1.25}
	\draw[->,yshift=0.45*\x cm] (\x.35,\y)--(\x.85,\y+0.25);
\else%
\fi}
\end{tikzpicture}
\caption{A translation path in $\HE{\mathbb{R}^2}$ visualized in $\mathbb{R}^2$.}
\end{figure}

Now we prove a lemma that will be very helpful for our next result.

\begin{lema} \label{dlem}
 Given two rectangles $A$ and $C$ with $A \subsetneq C$ there exist vertices $x$ of $A$ and $y$ of $C$ such that
 \begin{equation}
 \label{deqxy}h(A,C)=d(x,y).
 \end{equation}
\end{lema}

\begin{proof}
Suppose that $A$ is defined by the vertices $a$ and $b$, and $C$ by $c$ and $d$. First of all, note that $\ol{d}(A,C) = 0$ and if $y \in C \cap A$, then $d(c,A)=0$. So, we must have $y \in C \setminus A$. In particular there exist two sets of indices $J_1, J_2 \subset \{1,\ldots ,n\}$ such that if $j \in J_1$, then
 $$ c^j < y^j \leq a^j$$
 and if $k \in J_2$, then
 $$b^k < y^k \leq d^k.$$
 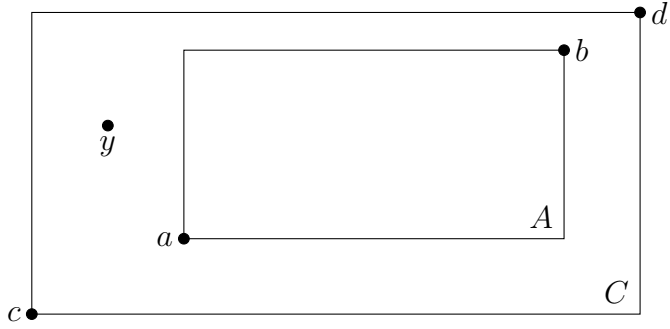
\begin{figure}[H]
 \centering
 \begin{tikzpicture}
  \draw (0,0)rectangle(8,4);
  \draw (2,1)rectangle(7,3.5);
  \draw[color=black,fill=black] (0,0) circle (2pt) node[left] {$c$};
  \draw[color=black,fill=black] (8,4) circle (2pt) node[right] {$d$};
  \draw[color=black,fill=black] (2,1) circle (2pt) node[left] {$a$};
  \draw[color=black,fill=black] (7,3.5) circle (2pt) node[right] {$b$};
  \draw[color=black,fill=black] (1,2.5) circle (2pt) node[below] {$y$};
  \node[above left] at (8,0) {$C$};
  \node[above left] at (7,1) {$A$};
 \end{tikzpicture}
 \caption{An illustration of a possible choice of $y$. In this case $1 \in J_1$ and $2 \not\in J_1 \cup J_2$.}
 \end{figure}
 \noindent Because $y \not\in A$, at least one of them is nonempty. Therefore, if we want to find $x \in A$ that is closest to $y$, it is reasonable to take $x_y$ such that:
 \begin{itemize}
  \item $x^i = a^i$, if $i \in J_1$;
  \item $x^i = b^i$, if $i \in J_2$;
  \item $x^i = y^i$, if $i \not\in J_1$ e $i \not\in J_2$.
 \end{itemize}
 In this case
 $$d(y,A) = d(y,x_y) = \left[\sum_{i \in J_1} (y^i - a^i)^2 + \sum_{i \in J_2} (y^i - b^i)^2\right]^{\frac{1}{2}}$$
 since there cannot be another element of $A$ closer to $y$. Since we want to maximize $d(y,A)$ over $y \in C \setminus A$, that is, we want to make the sum inside the square root as large as possible, it suffices to choose $y$ such that $\{1,\ldots,n\}=J_1 \cup J_2$ and with associated $x_y$ satisfying
 $$|y^i - x^i|= \max \{a^i-c^i,d^i-b^i\},$$
 which means that $y^i=c^i$ or $y^i=d^i$ for each $i=1,\ldots,n$. In other words, $y$ is a vertex of $C$ and by construction $x_y$ is a vertex of $A$. With this, we conclude that $h(A,C)=d(x,y)$.

\end{proof}

\begin{lema} \label{dtr}
Let $a,m,M \in \Rn$ with $m \ne M$ and such that the rectangle defined by $m$ and $M$ contains $a$. The map $f_{a,m,M}: [0,1] \to \HE{\Rn}$ defined by
$$f_{a,m,M}(t) = \prod_{i=1}^{n}[(m^i-a^i)t+a^i,(M^i-a^i)t+a^i]$$
is a path in $\HE{\Rn}$ from $\{a\}$ to the rectangle defined by $m$ and $M$.
\end{lema}

\begin{proof}
Notice that $f_{a,m,M}$ is well defined since $f_{a,m,M}(t)$ is a rectangle for every $t$ hence compact. In order to prove continuity again we show that $f_{a,m,M}$ is uniformly continuous. Assume $t_2 > t_1$ and denote $T_j=f_{a,m,M}(t_{j})$ for $j=1,2$. We have $T_1 \subset T_2$ since
{$$[(m^i-a^i)t_1+a^i,(M^i-a^i)t_1+a^i] \subset [(m^i-a^i)t_2+a^i,(M^i-a^i)t_2+a^i],$$}
for all $i$. (Notice that $a \in f_{a,m,M}(t)$ for all $t$, so $m^i-a^i \leq0$ and $M^i-a^i\geq 0$ for all $i$.)

\begin{figure}[H]
\centering
\begin{tikzpicture}[scale=1.4]
\draw[fill=black] (0,0) circle (1.5pt) node[below] {$a$};
\draw[fill=black] (1,1) circle (1.5pt) node[right] {$M$};
\draw[fill=black] (-6,-4) circle (1.5pt) node[left] {$m$};
\draw (1,1) rectangle (-6,-4);
\foreach \t in {0.25,0.5,0.8,1}{
\draw[dash pattern= on 2pt off 2pt] (1*\t,1*\t) rectangle (-6*\t,-4*\t);}
\draw[|-|,yshift=-4pt] (-3,-2)--(0.5,-2);
\node[below,yshift=-6pt] at (-1.25,-2) {\small$(M_1-m_1)t_1$};
\draw[|-|,yshift=-4pt] (-4.8,-3.2)--(0.8,-3.2);
\node[below,yshift=-6pt] at (-2,-3.2) {\small$(M_1-m_1)t_2$};
\draw[|-|,xshift=-4pt] (-3,-2)--(-3,0.5);
\node[below,rotate=-90,yshift=-6pt] at (-3,-0.75) {\small$(M_2-m_2)t_1$};
\draw[|-|,xshift=-4pt] (-4.8,-3.2)--(-4.8,0.8);
\node[below,rotate=-90,yshift=-6pt] at (-4.8,-1.2) {\small$(M_2-m_2)t_2$};
\draw[|-|,yshift=-4pt] (0.5,-4)--(0.8,-4);
\node[below,yshift=-6pt] at (0.65,-4) {\small$(M_1-a_1)(t_2-t_1)$};
\draw[|-|,yshift=-4pt] (-4.8,-4)--(-3,-4);
\node[below,yshift=-6pt] at (-3.9,-4) {\small$(m_1-a_1)(t_2-t_1)$};
\draw[|-|,xshift=-4pt] (-6,0.5)--(-6,0.8);
\node[left,xshift=-6pt] at (-6,0.65) {\small$(M_2-a_2)(t_2-t_1)$};
\draw[|-|,xshift=-4pt] (-6,-3.2)--(-6,-2);
\node[left,xshift=-6pt] at (-6,-2.6) {\small$(m_2-a_2)(t_2-t_1)$};
\end{tikzpicture}
\caption{Map $f_{a,M,m}(t)$ visualized in $\mathbb{R}^2$ for some values of $t$.}
\end{figure}
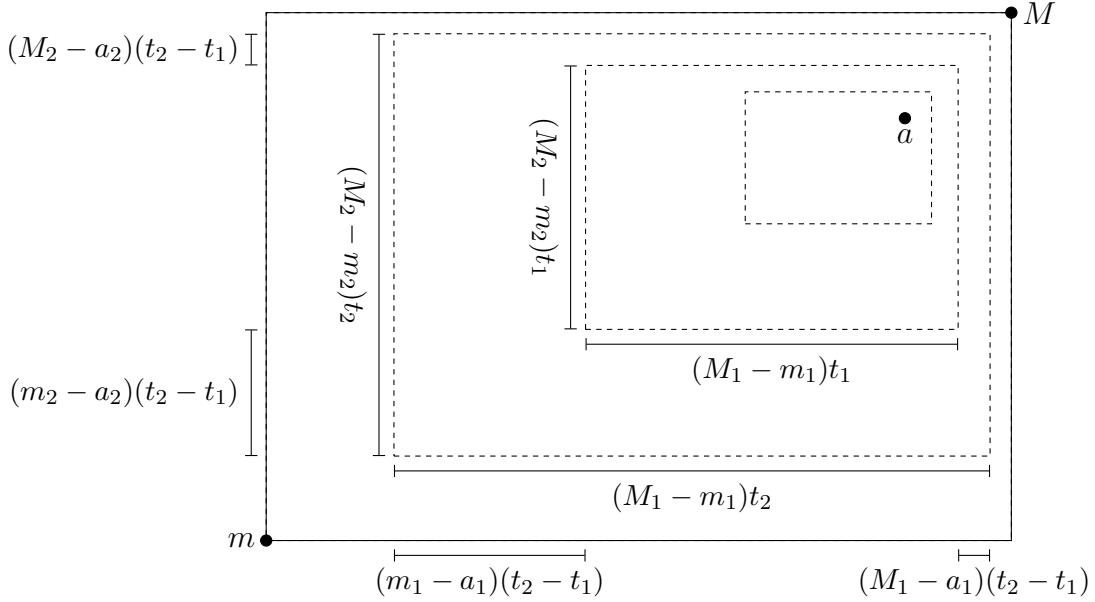

\noindent From the last lemma we have $h(T_1,T_2)=d(x,y)$, where $x$ and $y$ are vertices of $T_1$ and $T_2$, respectively, such that
$$|x^i - y^i|=\max \{(t_2-t_1)(M^i-a^i),(t_2-t_1)(a^i-m^i)\}.$$
Setting $$S=\max \left\{ \max_i |m_{i}-a_{i}|, \max_i |M_{i}-a_{i}| \right\}$$
we have
\begin{align*}
d(x,y) &= \left[ \sum_{i=1}^n \max \left\{ [(m_{i}-a_{i})(t_2-t_1)]^2,[(M_{i}-a_{i})(t_2-t_1)]^2\right\} \right]^{\frac{1}{2}} \\
& \leq \left[\sum_i^n S^2(t_2-t_1)^2\right]^\frac{1}{2} = n^\frac{1}{2}S(t_2-t_1).
\end{align*}
Thus, given $\varepsilon >0$, if we take $\delta = \dfrac{\varepsilon}{n^\frac{1}{2}S}$, for all $t_1,t_2 \in [0,1]$ such that $|t_1-t_2|<\delta$ there holds $$h(f(t_1),f(t_2))=h(T_1,T_2)=d(x,y)<\varepsilon.$$

\end{proof}

Now we show that any compact set is pathwise connected to any rectangle in which it is contained.

\begin{lema} \label{ctr}
Let $A$ be a compact subset of $\Rn$ and $m,M \in \Rn$ with $m \ne M$ such that the rectangle defined by $m$ and $M$ contains $A$. The map $f_{A,m,M}: [0,1] \to \HE{\Rn}$ defined by
$$f_{A,m,M}(t) = \bigcup_{a \in A} f_{a,m,M}(t)$$
is a path in $\HE{\Rn}$ from $A$ to the rectangle of vertices $m$ and $M$.
\end{lema}

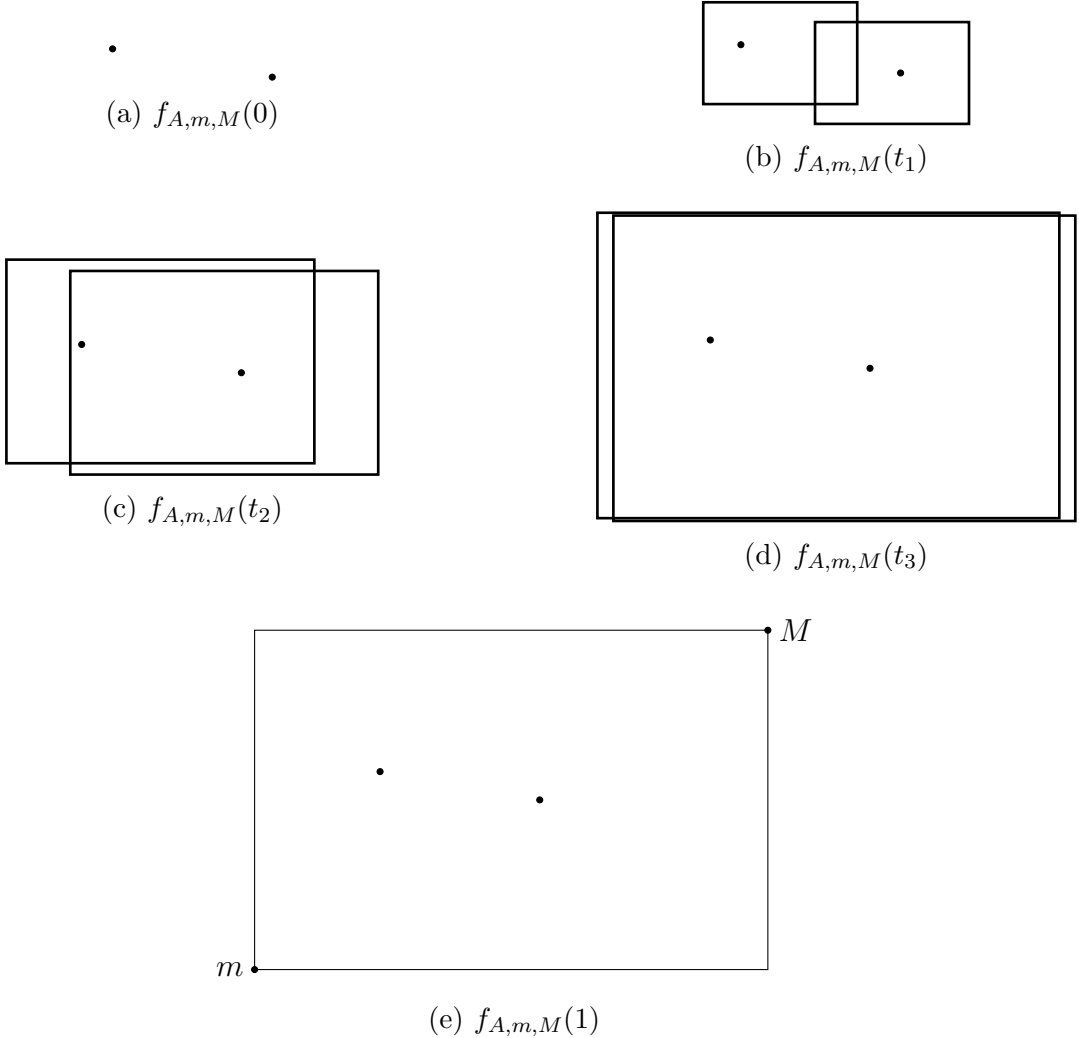
\begin{figure}[H]
\centering
\begin{subfigure}{0.49\linewidth}
\centering
\begin{tikzpicture}[scale=0.75]
\draw[color=black,fill=black] (-1.8,0.5) circle (1.5pt);
\draw[color=black,fill=black] (1,0) circle (1.5pt);
\end{tikzpicture}
\subcaption{$f_{A,m,M}(0)$}
\end{subfigure}
\begin{subfigure}{0.49\linewidth}
\centering
\begin{tikzpicture}[scale=0.75]
\foreach \x/\y in {-1.8/0.5,1/0}{
\foreach \t/\cor in {0.30/red}{
\pgfmathsetmacro\a{5*\t+\x*(1-\t)};
\pgfmathsetmacro\b{(3-\y)*\t+\y};
\pgfmathsetmacro\c{\x-(\x+4)*\t};
\pgfmathsetmacro\d{-(\y+3)*\t+\y};
\draw[color=black,line width=1pt] (\a,\b)--(\c,\b)--(\c,\d)--(\a,\d)--cycle;}}
\draw[color=black,fill=black] (-1.8,0.5) circle (1.5pt);
\draw[color=black,fill=black] (1,0) circle (1.5pt);
\end{tikzpicture}
\subcaption{$f_{A,m,M}(t_1)$}
\end{subfigure}

\vspace{12pt}
\begin{subfigure}{0.49\linewidth}
\centering
\begin{tikzpicture}[scale=0.75]
\draw[color=black,fill=black] (-1.8,0.5) circle (1.5pt);
\draw[color=black,fill=black] (1,0) circle (1.5pt);
\foreach \x/\y in {-1.8/0.5,1/0}{
\foreach \t/\cor in {0.60/thirdblue}{
\pgfmathsetmacro\a{5*\t+\x*(1-\t)};
\pgfmathsetmacro\b{(3-\y)*\t+\y};
\pgfmathsetmacro\c{\x-(\x+4)*\t};
\pgfmathsetmacro\d{-(\y+3)*\t+\y};
\draw[color=black,line width=1pt] (\a,\b)--(\c,\b)--(\c,\d)--(\a,\d)--cycle;}}
\end{tikzpicture}
\subcaption{$f_{A,m,M}(t_2)$}
\end{subfigure}
\begin{subfigure}{0.49\linewidth}
\centering
\begin{tikzpicture}[scale=0.75]
\draw[color=black,fill=black] (-1.8,0.5) circle (1.5pt);
\draw[color=black,fill=black] (1,0) circle (1.5pt);
\foreach \x/\y in {-1.8/0.5,1/0}{
\foreach \t/\cor in {0.90/fifthblue}{
\pgfmathsetmacro\a{5*\t+\x*(1-\t)};
\pgfmathsetmacro\b{(3-\y)*\t+\y};
\pgfmathsetmacro\c{\x-(\x+4)*\t};
\pgfmathsetmacro\d{-(\y+3)*\t+\y};
\draw[color=black,line width=1pt] (\a,\b)--(\c,\b)--(\c,\d)--(\a,\d)--cycle;}}
\end{tikzpicture}
\subcaption{$f_{A,m,M}(t_3)$}
\end{subfigure}

\vspace{12pt}
\begin{subfigure}{\textwidth}
\centering
\begin{tikzpicture}[scale=0.75]
\draw (5,3)--(-4,3)--(-4,-3)--(5,-3)--cycle;
\draw[color=black,fill=black] (-4,-3) circle (1.5pt) node[left] {$m$};
\draw[color=black,fill=black] (5,3) circle (1.5pt) node[right] {$M$};
\draw[color=black,fill=black] (-1.8,0.5) circle (1.5pt);
\draw[color=black,fill=black] (1,0) circle (1.5pt);
\end{tikzpicture}
\subcaption{$f_{A,m,M}(1)$}
\end{subfigure}
\caption{Mapping $f_{A,m,M}(t)$ when $A$ is a set of two points. For each $t$, $f_{A,m,M}(t)$ is the union of two rectangles. In this figure we have $0 <t_1<t_2<t_3<1$.}
\end{figure}

\begin{proof}
 The proof will be done in two steps. In the following, we write $f_{a}$ in place of $f_{a,m,M}$ and $f_{A}$ in place of $f_{A,m,M}$.

 \ul{\textit{Step 1}} - Show that $f_{A,m,M}(t)$ is compact for every $t \in [0,1]$.

For $t=0$ and $t=1$ $f_{A}(t)$ is trivially compact since it is a rectangle. Fixed $t$, $f_{A}(t)$ is a bounded set in $\Rn$ since for each $a \in A$ the set $f_a(t)$ is contained in $f_a(1)=f_A(1)$, therefore $f_A(t) \subset f_A(1)$ which is a rectangle. Thus, we just need to prove that $f_{A}(t)$ is closed. To this end, let $(x_n)$ be a sequence contained in $f_{A}$ converging to $x$. Since $(x_n) \subset f_{A}(t)$, for every $n$ there is a $a_n \in A$ such that $x_n \in f_{a_n}(t)$. This gives us a sequence $(a_n) \subset A$ and from the compactness of $A$ there is a subsequence, which we will continue to denote by $(a_n)$ in order to simplify the notation. We claim that $x \in f_{a}(t)$ and consequently $x \in f_{A}(t)$.

\begin{figure}[H]
\centering
\begin{tikzpicture}[scale=2.5,line width=1pt]
\foreach \x in {0,15,28,39,48,55}{
\draw[dash pattern=on 2pt off 2pt] (-2+0.\x,-1+0.0\x) rectangle (2+0.\x,1+0.0\x);
\draw[fill=black] (0.\x,0.0\x) circle (1pt);
\draw[color=green!70!black,fill=green!70!black] (-0.2,0.40+0.\x) circle (1pt);
}
\draw[color=blue,fill=blue] (0.63,0.063) circle (1pt) node[below] {\color{black}$a$};
\draw[color=blue] (-2+0.63,-1+0.063) rectangle (2+0.63,1+0.063);
\draw[color=red,fill=red] (-0.2,1.055) circle (1pt) node[above]{\color{black}$x$};
\end{tikzpicture}
\caption{In green we have a sequence $(x_n)$ and in black a sequence $(a_n)$ in $A$ (not depicted). The dashed rectangles represent some of the $f_{a_n}(t)$. From the figure it is natural to think that $x \in f_a(t)$.}
\end{figure}
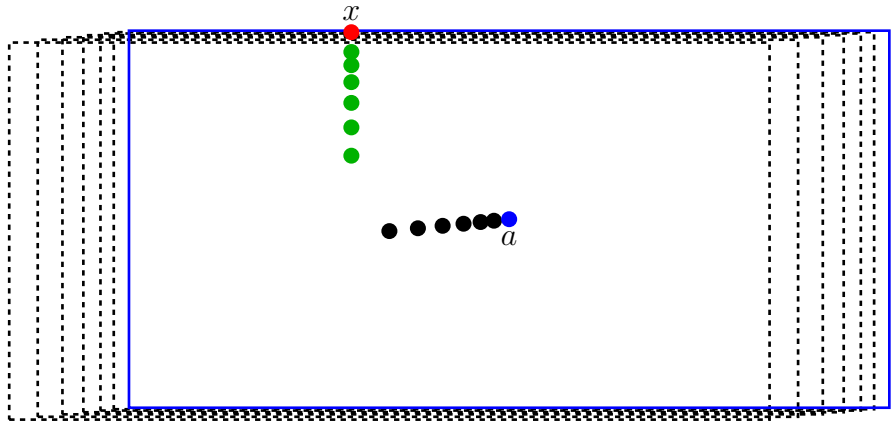

Suppose that $x \not\in f_{a}(t)$. Since $f_a(t)$ is a closed set, there is an open ball $B(x,\delta) \subset \Rn \setminus f_{a}(t)$. We will find an element of $f_{a}(t)$ in this ball obtaining a contradiction.

\begin{figure}[H]
\centering
\begin{tikzpicture}[scale=2.5,line width=1pt]
\foreach \x in {0,15,28,39,48,55}{
\draw[dash pattern=on 2pt off 2pt] (-2+0.\x,-1+0.0\x) rectangle (2+0.\x,1+0.0\x);
\draw[fill=black] (0.\x,0.0\x) circle (1pt);
\draw[color=green!70!black,fill=green!70!black] (-0.2,0.40+0.\x) circle (1pt);
}
\draw[color=blue,fill=blue] (0.63,0.063) circle (1pt) node[below] {\color{black}$a$};
\draw[color=blue,yshift=-12pt] (-2+0.63,-1+0.063) rectangle (2+0.63,1+0.063);
\draw (-0.2,1.055) circle (11pt);
\draw[color=red,fill=red] (-0.2,1.055) circle (1pt) node[above]{\color{black}$x$};
\end{tikzpicture}
\caption{The situation assumed in the contradiction.}
\end{figure}
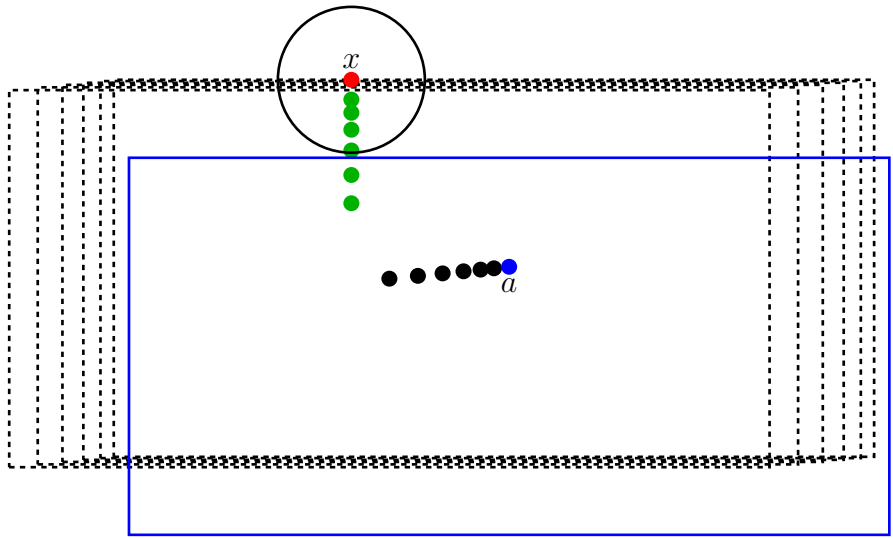

Given $\varepsilon>0$, let $N \in \mathbb{N}$ be such that if $n \geq N$, then $d(a_n,a) < \varepsilon$. We will estimate $h(f_a(t),f_{a_n}(t))$ in terms of $\varepsilon$. First, the sets $f_{a_n}(t)$ and $f_a(t)$ are
\begin{align*}
f_{a_n}(t) = \prod_{i=1}^{n}[(m^i-a_n^i)t+a_n^i,(M^i-a_n^i)t+a_n^i]\\
f_{a}(t) = \prod_{i=1}^{n}[(m^i-a^i)t+a^i,(M^i-a^i)t+a^i].
\end{align*}
In order to find an upper bound for the Hausdorff distance, we will define two special rectangles for each $n$. One of them will be denoted by $W_n$ and is determined by the vertices $Q_n$ and $R_n$, where
\begin{align*}
R_{n}^i&= \max\{(M^i-a_n^i)t+a_n^i,(M^i-a^i)t+a^i\}
\end{align*}
and
\begin{align*}
Q_{n}^{i}&= \min\{(m^i-a_n^i)t+a_n^i,(m^i-a^i)t+a^i\},
\end{align*}
while the other one will be denoted by $w_n$ and is determined by the vertices $r_n$ and $q_n$, where
\begin{align*}
r_{n}^i&= \min\{(M^i-a_n^i)t+a_n^i,(M^i-a^i)t+a^i\}
\end{align*}
and
\begin{align*}
q_{n}^i&= \max\{(m^i-a_n^i)t+a_n^i,(m^i-a^i)t+a^i\}.
\end{align*}

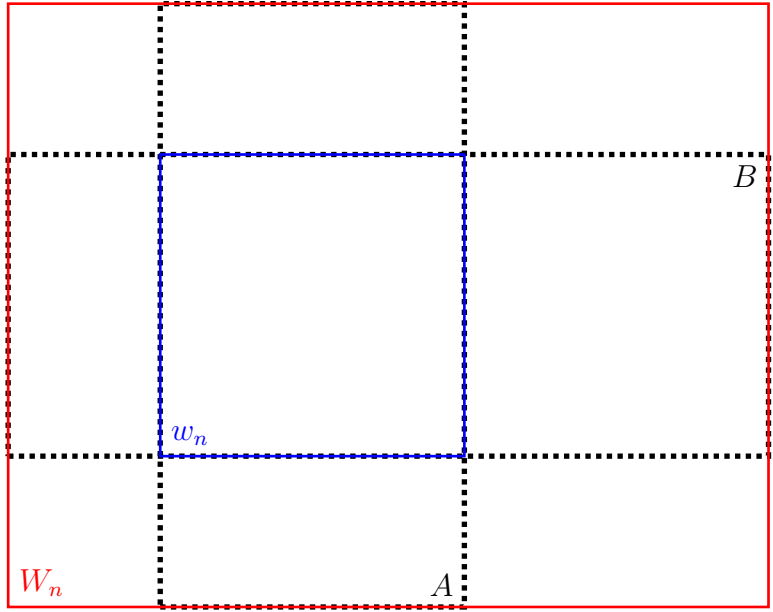
\begin{figure}[H]
\centering
\begin{tikzpicture}[scale=2]
\draw[dash pattern=on 2pt off 2pt,line width=2pt] (-2,-1) rectangle (3,1);
\draw[dash pattern=on 2pt off 2pt,line width=2pt] (-1,-2) rectangle (1,2);
\draw[color=blue,line width=1pt] (-1,-1) rectangle (1,1);
\draw[color=red,line width=1pt] (-2,-2) rectangle (3,2);
\node[above right,color=red] at (-2,-2) {$W_n$};
\node[above right,color=blue] at (-1,-1) {$w_n$};
\node[above left] at (1,-2) {$A$};
\node[below left] at (3,1) {$B$};
\end{tikzpicture}
\caption{Example of how rectangles $A$ and $B$ generate the rectangles $W_n$ and $w_n$ in $\mathbb{R}^2$.}
\end{figure}

\noindent Note that we have $w_n \subset f_{a}(t) \cap f_{a_n}(t) \subset f_a(t) \cup f_{a_n}(t) \subset W_n$, thus $h(f_a(t),f_{a_n}(t)) \leq h(w_n,W_n)$ and in particular  both $|R_{n}^i-r_{n}^i|$ and $|Q_{n}^i-q_{n}^i|$ are equal to $(1-t)|a^i-a_n^i|$. Since $w_n \subset W_n$, by Lemma \ref{dlem} we obtain our estimate
\begin{align*}
h(w_n,W_n) &= \left( \sum_{i=1}^n \max \{|R_{n}^i-r_{n}^i|^2,|Q_{n}^i-q_{n}^i|^2\}\right)^{\frac{1}{2}} \\
&= \left(\sum_{i=1}^n (1-t)^2|a^i-a_n^i|^2\right)^{\frac{1}{2}} \\
&= (1-t)\left(\sum_{i=1}^n |a^i-a_n^i|^2\right)^{\frac{1}{2}} \\
&= (1-t)d(a,a_n)\\
&< (1-t)\varepsilon.
\end{align*}

Thus, choosing $\varepsilon = \dfrac{\delta}{2(1-t)}$ for $n \geq N$, we have
$$h(f_a(t),f_{a_n}(t)) \leq h(w_n,W_n) < \dfrac{\delta}{2}.$$
This means that for $x_n \in f_{a_n}(t)$ there is some $b \in f_a(t)$ such that $d(x_n,b) < \delta /2$. Since $(x_n) \to x$, there is $N' \in \mathbb{N}$ such that if $n \geq N'$  , then $d(x_n,x) < \delta/2$. Taking $N_0 = \max \{N,N'\}$, if $n \geq N_0$ then
$$d(x,b) \leq d(x,x_n)+d(x_n,b) < \dfrac{\delta}{2} + \dfrac{\delta}{2} = \delta$$
so $$B(x,\delta)\cap f_{a}(t) \ne \emptyset,$$
the expected contradiction. We now have that $f_a(t)$ is closed and hence a compact set.

\ul{Step 2} - Show that $f_{A}(t)$ is a path, that is, it is continuous. Indeed it is uniformly continuous.

We will use the uniform continuity of $f_{a}(t)$. As we saw in the proof of Lema \ref{dtr}, given $\varepsilon > 0$ and $a \in A$, there exists $\delta(a)=\dfrac{\varepsilon}{n^\frac{1}{2}S_a}$ such that if $|t_1 - t_ 2|<\delta$ then $h(f_{a}(t_1),f_{a}(t_2))< \varepsilon$, where
$$S_a=\max \left\{ \max_i \ |m_{i}-a_{i}|, \max_i |M_{i}-a_{i}| \right\}.$$
Set
$$S=\sup_{a\in A} S_a = \max_{a \in A} S_a \leq \max_i {|M^i-m^i|}$$
where the last inequality is due to the fact that $$m^i \leq a^i \leq M^i$$
for all $i$ since $a$ is in the rectangle defined by $m$ and $M$. Therefore, if $|t_1-t_2|<\delta = \dfrac{\varepsilon}{nS}$, we have
$$h(f_{a}(t_1),f_{a}(t_2)) < \varepsilon \text{ for all } a \in A,$$
since $\delta < \delta(a)$ for all $a \in A$.

Now,
$$h(f_{A}(t_1),f_{A}(t_2)) = \max \{\ol{d}(f_{A}(t_1),f_{A}(t_2)),\ol{d}(f_{A}(t_2),f_{A}(t_1))\}$$
so we can assume without loss of generality that $t_1 < t_2$ and just analyze $\ol{d}(f_{A}(t_2),f_{A}(t_1))$ since $\ol{d}(f_{A}(t_1),f_{A}(t_2))$ will be zero. From the definition,
\begin{equation*}\ol{d}(f_{A}(t_2),f_{A}(t_1))=\ol{d}\left(\bigcup_a f_{a}(t_2),\bigcup_a f_{a}(t_1)\right) = \max_{x \in \underset{a \in A}{\bigcup}f_a(t_2)} d\left(x, \bigcup_a f_a(t_1) \right).\end{equation*}
Since $x \in f_{a}(t_2)$ for some $a \in A$, $f_a(t_1) \subset \cup_a f_a(t_1)$ and $h(f_{a}(t_1),f_{a}(t_2)) < \varepsilon$ for all $a \in A$, we have
$$d\left(x, \bigcup_a f_a(t_1) \right) < d\left(x, f_a(t_1) \right) \leq \max_{x \in f_a(t_2)} d\left(x, f_a(t_1) \right) = \ol{d}(f_a(t_2),f_a(t_1)) < \varepsilon.$$
Intuitively, this distance is smaller then $\varepsilon$ because for every $x \in \cup_a f_a(t_2)$ there is a point $y$ in some $f_{a}(t_1)$ whose distance to $x$ is smaller than $\varepsilon$. Thus, $h(f_{A}(t_1),f_{A}(t_2)) = \ol{d}(f_{A}(t_2),f_{A}(t_1)) < \varepsilon$ and $f$ is a path since it is uniformly continuous.

\end{proof}

 Notice that in the process we have shown that a particular uncountable union of compact sets is compact, which is particularly interesting.

\begin{teorema}
 $\HE{\Rn}$ is pathwise connected.
 \end{teorema}

 \begin{proof}
 Let $A$ and $B$ be elements of $\HE{\Rn}$. Since they are compact subsets of $\Rn$, we can find two rectangles $A_R$ and $B_R$ such that $A \subset A_R$, $B \subset B_R$ and $B_R$ is a translation of $A_R$. Putting together everything we made, we know that there exist
 \begin{enumerate}[label={\arabic*)}]
  \item a path $f:[0,1] \to \HE{\Rn}$ such that $f(0)=A$ and $f(1)=A_R$ by Lemma \ref{ctr};
  \item a path $g:[0,1] \to \HE{\Rn}$ such that $g(0)=A_R$ and $g(1)=B_R$ by Lemma \ref{translationpath};
  \item a path $h:[0,1] \to \HE{\Rn}$ such that $h(0)=B$ and $h(1)=B_R$ by Lemma \ref{ctr}.
 \end{enumerate}
Using the parametrizations
\begin{align*}
\alpha: [0,1] \to [0,3], \alpha(t)=3t; \\
\gamma:[1,2] \to [0,1], \gamma(t)=t-1;\\
\gamma_0:[2,3] \to [0,1], \gamma(t)= 3-t,
\end{align*}
we obtain a path $F:[0,1] \to \HE{\Rn}$ from $A$ to $B$ defining
$$
F(t)= \begin{cases}
       (f \circ \alpha)(t), \text{ se } 0\leq t \leq \dfrac{1}{3} \\ \\
       (g \circ \gamma \circ \alpha)(t), \text{ se } \dfrac{1}{3} < t \leq \dfrac{2}{3} \\ \\
       (h \circ \gamma_0 \circ \alpha)(t), \text{ se } \dfrac{2}{3} < t \leq 1.
      \end{cases}
$$

\end{proof}

\section*{Acknowledgments}
The first author was supported by the Conselho Nacional de Desenvolvimento Científico e Tecnológico (CNPq) under grant 163467/2021-8.

\end{document}